\documentclass[letterpaper, 10 pt, conference]{ieeeconf}
\usepackage{array} 
\usepackage{stfloats}   
\usepackage{mathtools}
\usepackage{xcolor}
\usepackage[utf8]{inputenc}
\usepackage{tabularx}
\usepackage{booktabs} 
\usepackage{mathrsfs}
\usepackage{graphics} 

\usepackage{amsmath} 
\usepackage{amssymb}  
\usepackage{amsthm}
\usepackage{cite}
\usepackage{bm}
\usepackage{acronym}
\usepackage{paralist}
\usepackage{float}
\usepackage[dvipsnames]{xcolor}
\usepackage{epstopdf}
\usepackage{multicol}
\usepackage{tikz}
\usepackage{hyperref}
\usepackage{graphicx}
\usepackage{soul}
\usepackage{macros} 

\usepackage{anyfontsize}

\usepackage{stfloats}
\fnbelowfloat

\makeatletter
\hypersetup{colorlinks=true}
\AtBeginDocument{\@ifpackageloaded{hyperref}
  {\def\@linkcolor{blue}
  \def\@anchorcolor{red}
  \def\@citecolor{red}
  \def\@filecolor{red}
  \def\@urlcolor{black}
  \def\@menucolor{red}
  \def\@pagecolor{red}
\begingroup
  \@makeother\`%
  \@makeother\=%
  \edef\x{%
    \edef\noexpand\x{%
      \endgroup
      \noexpand\toks@{%
        \catcode 96=\noexpand\the\catcode`\noexpand\`\relax
        \catcode 61=\noexpand\the\catcode`\noexpand\=\relax
      }%
    }%
    \noexpand\x
  }%
\x
\@makeother\`
\@makeother\=
}{}}
\makeatother

\DeclareMathOperator{\divergence}{\mathrm{div}}
\DeclareMathOperator{\grad}{\mathrm{grad}}
\DeclareMathOperator{\interior}{\mathrm{int}}

\newcommand{\bequ}{\begin{eqnarray}}
\newcommand{\eequ}{\end{eqnarray}}

\IEEEoverridecommandlockouts

\def\BibTeX{{\rm B\kern-.05em{\sc i\kern-.025em b}\kern-.08em
    T\kern-.1667em\lower.7ex\hbox{E}\kern-.125emX}}

\definecolor{AA}{RGB}{34,139,34}

\definecolor{HM}{RGB}{200,00,00}

\begin{document}
\title{\LARGE \bf 
A Geometric Solution of the Schrödinger Bridge Problem on $\SO2$ via Stochastic Optimal Control

\thanks{This research has been supported by NJIT's startup funds.}
\thanks{The authors are with the 
Department of Mechanical \& Industrial Engineering at the New Jersey Institute of Technology, Newark, NJ 07102, USA (email: \texttt{\{hm576, adeel.akhtar\}@njit.edu}). 
}}
\author{Hamza Mahmood \and Adeel Akhtar}
\maketitle
\begin{abstract}
We present a geometric coordinate-free solution to the isotropic Schrödinger bridge problem (SBP) for the kinematic equation on the Lie group $\SO2$. We consider the angular velocity of the system as the control input and assume that the given initial and terminal state probability density functions defined on $\SO2$ in our SBP are continuous and strictly positive. We solve the SBP by proving the existence and uniqueness of a solution to the so-called Schrödinger system of equations on $\SO2$, by showing that a fixed-point recursion is contractive in a complete metric space with respect to the Hilbert's projective metric. The geometric controller thus designed only uses the intrinsic geometric structure of $\SO2$ and does not embed it in the Euclidean plane to achieve the optimal density control. The numerical simulation verifies the validity of the theoretical construction of the Schrödinger bridge. The code and animations are publicly available at \texttt{\href{https://gitlab.com/a5akhtar/sbp}{https://gitlab.com/a5akhtar/sbp}}.
\end{abstract}

\section{Introduction}
\label{sec:Introduction}
The Schrödinger bridge problem (SBP) is a stochastic optimal control problem that requires designing a controller that drives the controlled joint state probability density function (PDF) of a stochastic dynamical system from a prescribed initial PDF to a prescribed terminal PDF over a finite-time horizon while minimizing the expected control effort needed to achieve this steering. This density control is different from the usual control objective in the sense that the aim is to optimally steer the evolution of the state PDF rather than an individual state trajectory. Solving the SBP for stochastic systems finds applications in controlling physical populations such as robotic swarms~\cite{BanChuHad2017} and highway traffic densities~\cite{ChiZhaIoa1997}. 

Most robotic systems, such as drones and wheeled robots, have a Lie group as their state space. However, the SBPs considered in the existing literature predominantly address systems evolving on a Euclidean space, or they consider systems with state space embedded in a Euclidean space~\cite{CalHal2020}. The main aim of this paper is to design a geometric controller that solves the SBP for a system evolving on a particular compact Lie group, namely $\SO2$.

The subject of SBP is studied in the case of no prior dynamics in~\cite{Jam1975}, which is what is referred to as the ``classical SBP", and with prior dynamics in~\cite{CheGeoPav2016}. In these works, the density control is studied on Euclidean space. To the best of our knowledge, the SBP for systems on compact Lie groups has not been studied geometrically. One may argue that for an SBP for a system on a compact Lie group, one can embed the Lie group in some Euclidean space and employ the methods in the literature available to solve the SBP on that Euclidean space. However, designing a controller for an SBP on a compact Lie group while only working with the intrinsic geometric Lie group structure is an interesting problem in itself, which appears to be not yet studied in the literature. 
Studying an SBP associated with a stochastic system on a Lie group in a geometric setting offers a challenging problem with several practical applications. An example of such a geometric SBP, namely, the SBP for the kinematic equation on $\SO2$, is the main focus of this paper. 

In this article, we consider the stochastic differential equation (SDE) for the kinematic equation on the Lie group $\SO2$, that is, the kinematic equation together with Brownian noise. The SDE studied is in the Stratonovich sense~\cite{hsu2002}. 
The control input here is the angular velocity of the system with the controlled state lying on the Lie group. Consequently, the corresponding PDFs are defined on $\SO2$, and Haar measure is used to define the integration of functions, including PDFs, over $\SO2$. 
The SBP for the kinematic equation on $\SO2$ then becomes a problem of designing a geometric controller for the angular velocity of the system such that the controlled state PDF of the system moves from a given initial PDF on $\SO2$ to a given terminal PDF while also minimizing the control effort to do so.  
{To solve this challenging stochastic optimal control problem, our solution approach uses the idea of what} is called the Hilbert's projective metric~\cite{Bus1973} defined using a closed solid cone in a Banach space.  
Our solution leads to the geometric and coordinate-free controller for our SBP, and we underscore that our design uses the intrinsic geometric structure of $\SO2$ and does not embed it in the Euclidean plane.

Our main contributions are the following: (i) the proof of the existence and uniqueness of a solution to the Schrödinger system on $\SO2$ using the Hilbert's projective metric (Theorem~\ref{thm:Soln_Schrödinger_sys_SO2}) and (ii) consequently, a geometric coordinate-free solution \eqref{eq:Omega^{opt}_soln} to the isotropic SBP (Problem~\ref{problem:SBP_on_SO(2)}) for the kinematic equation on $\SO2$.

\subsection{Notation} 
\label{subsec:Notation} 

The set $\Real^n$ denotes the Euclidean space of dimension $n$. We write the set of all nonnegative real numbers as $\mathbb{R}_{\geq 0}$. The notations $|\cdot|$ and $\exp(\cdot)$ mean  the absolute value and the exponential of a real number, respectively. The Euclidean 2-norm of a vector in $\mathbb{R}^n$ is denoted with $\|\cdot\|_{2}$. The infimum (greatest lower bound) and the supremum (least upper bound) of a set $S \subset \mathbb{R}$ are written as $\inf S$ and $\sup S$, respectively. The interior of a subset $S$ of a topological space is denoted by $\interior S$. The composition of a map $F : X \to X$ with itself $k$ times is represented by $F^k$, i.e., ${F}^{k} = {F} \circ {F} \circ ... \circ {F}$ ($k$ times). The supremum norm $\|\cdot\|_{\infty}$ on the space $C(\SO2)$ of real-valued continuous functions on $\SO2$ is defined as $\|f\|_{\infty} \eqdef \sup_{R \in \SO2} |f(R)|$ for any $f \in C(\SO2)$. We denote by $B(f,r)$ a ball of radius $r>0$ centered at $f$ in the normed space $\left(C(\SO2),\|\cdot\|_{\infty}\right)$. The notation $f_{n} \xrightarrow[]{d_H} f$ means that the sequence $\{f_n\}$ converges to $f$ with respect to the Hilbert's projective metric $d_H$ defined in Definition~\ref{def:Hilbert_metric}, i.e., $\lim_{n \to \infty} d_H(f_n, f) = 0$. 
The remaining mathematical symbols used in this paper are defined when they first appear in their respective sections.

\section{Preliminaries}
\label{sec:Preliminaries}
In this section, we introduce the necessary concepts and key definitions needed to formulate our Schrödinger bridge problem (SBP). We refer the reader to~\cite{CalHal2022} for the background on classical SBP on $\mathbb{R}^n$.
\subsection{Kinematic Equation on \texorpdfstring{$\SO2$}{SO(2)}}
Consider the Lie group $\SO2$, which is a special orthogonal group of $2 \times 2$ real orthogonal matrices.
The Lie algebra of $\SO2$, denoted $\mathfrak{so}(2)$, consists of all $2 \times 2$ real skew-symmetric matrices.
Define a map $\hat{\cdot} : \mathbb{R} \to \mathfrak{so}(2)$ (called the hat map) by
\begin{equation}
\label{eq:hat_map}
\mathbb{R} \ni \Omega \mapsto 
\widehat{\Omega} = \begin{bmatrix} 
0 & -\Omega \\ 
\Omega & 0 \end{bmatrix} \in \mathfrak{so}(2). 
\end{equation}
Let $(\cdot)^{\vee} : \mathfrak{so}(2) \to \mathbb{R}$ denote the inverse of the hat map. 
Let \(R(t)\in \SO2\) be the system's configuration at time $t$, \(\dot R(t)\in T_{R(t)}\SO2\) its velocity, and \(\Omega(R,t)\in\mathbb R\) the angular velocity. The kinematics are given by
\begin{equation}
\label{eq:kinematic_eq_on_SO2}
\dot{R}(t)=R(t)\,\widehat{\Omega}(R,t),
\end{equation}
which is coordinate-free. Since \(\widehat{\Omega}(R,t)\in \mathfrak{so}(2)\), it follows that \(R(t)\widehat{\Omega}(R,t)\in T_{R(t)}\SO2\); therefore, the trajectory remains on \(\SO2\) at all times.

\subsection{Stratonovich Stochastic Differential Equation (SDE) on a Compact Lie Group}
We use the definition of Stratonovich SDE on a Riemannian manifold from~\cite{hsu2002} to define Stratonovich SDE on a compact Lie group below. We refer the reader to~\cite{gallier2020} for the concept of an Ad-invariant inner product on a Lie algebra and the concept of the induced bi-invariant Riemannian metric on a Lie group.   
\begin{definition}[Stratonovich SDE on a Compact Lie Group]
Let $G$ be a compact Lie group with Lie algebra $\mathfrak g$. Consider an Ad-invariant inner product $\langle -,- \rangle_{\mathfrak{g}}$ on $\mathfrak{g}$ and let $G$ be equipped with the induced bi-invariant Riemannian metric. Fix $m \in \mathbb N$, and let $(W_t)_{t\ge0}$ be an $\mathbb R^m$-valued standard Brownian motion. 
For each $i=0, 1,\dots,m$, let $V_i$ be a smooth vector field on $G$. 
A $G$-valued continuous semimartingale $(R_t)_{t\ge 0}$ is said to solve the \emph{Stratonovich stochastic differential equation (SDE)}
\[
dR_t \;=\; V_0(R_t)\,dt \;+\; \sum_{i=1}^m V_i(R_t)\circ dW_t^i, 
\qquad R_0 = x \in G,
\]
if for every smooth test function $f\in C^\infty(G)$, we have  
\[
f(R_t) = f(R_0) + \int_0^t (V_0 f)(R_s)\,ds + \sum_{i=1}^m \int_0^t (V_i f)(R_s)\circ dW_s^i.
\] Here $Vf$ denotes the directional derivative of $f$ along the vector field $V$, and the symbol $\circ$ indicates the Stratonovich stochastic integral.
\end{definition}
This definition is intrinsic and coordinate-free: the driving vector fields $V_i$ are geometric objects on $G$, and the solution $R_t$ remains on $G$ for all $t$. 
The symbol $\circ$ in the expression $V(R_t)\circ dW_t$ indicates that the stochastic integral is taken in the \emph{Stratonovich sense};  
see~\cite{bernt2013}. The Stratonovich formulation of SDEs is preferred on manifolds and Lie groups because it respects the ordinary chain rule of differential geometry. See~\cite{hsu2002}. 

\subsection{Stratonovich SDE for the Kinematic Equation on \texorpdfstring{$\SO2$}{SO(2)} }
\begin{definition}[Stratonovich SDE for the Kinematics on $\SO2$]
Fix $m \in \mathbb N$, and let $(W_t)_{t\ge0}$ be an $\mathbb R^m$-valued standard Brownian motion.   
Let $\Omega: \SO2 \times[0,1]\to\mathbb R$ be an angular velocity field and let $\{\sigma_i : \SO2 \to \Real \}_{i=1}^m$ be smooth scalar noise amplitudes (or take $\sigma_i$ constant for invariant noise). The \emph{Stratonovich SDE for the kinematic equation on $\SO2$} is of the following form:
\[
{\qquad
dR_t \;=\; R_t\,\widehat{\Omega}(R_t,t)\,dt \;+\; \sum_{i=1}^m R_t\,\widehat{\sigma}_i(R_t)\;\circ dW_t^i,
\qquad}
\]
where $R_0\in\SO2$ and $\circ$ denotes the Stratonovich integral.
\end{definition}
\subsection{The Laplace-Beltrami Operator and Haar Measure}
\begin{definition}[Laplace--Beltrami Operator on a Compact Lie Group] 
\label{def:LapBeltOperator}
Let $G$ be a compact Lie group and $\mathfrak{g}$ the Lie algebra of $G$\hspace{0.5mm}. Consider an Ad-invariant inner product $\langle -,- \rangle_{\mathfrak{g}}$ on $\mathfrak{g}$ and let $G$ be equipped with the induced bi-invariant Riemannian metric.  Then, with $G$ viewed as a Riemannian manifold, the \emph{Laplace--Beltrami operator} $\Delta_G$ on $G$ is defined as follows:  for all $f \in C^{\infty}(G)$,\, 
$
\Delta_G f \;\eqdef\; \divergence ( \grad f ).$
\end{definition}

We next define the Haar measure with respect to which the integration of functions over $\SO2$ is performed. We refer the reader to~\cite{Faraut2008} for the details.
\begin{definition}[Left Haar measure] 
\label{def:Haar_measure}
Let $G$ be a locally compact Lie group. A (non-zero) Radon measure $\mu$ on $G$ is called a \emph{left Haar measure} (or simply \emph{Haar measure}) if it is left invariant, that is, for every Borel set $E \subset G$ and for every $g \in G$,
$
\mu(gE)=\mu(E),
$
where $gE \eqdef \{\, gx \, | \, x \in E \,\}$. 
\end{definition}

\subsection{Closed Solid Cone and the Hilbert's Projective Metric}

\begin{definition}[Closed solid cone in a Banach space {\cite[Sec.~2]{Bus1973}}]
\label{def:closed_solid_cone} Let $\mathcal{X}$ be a real Banach space. A subset $\mathcal{K}$ of $\mathcal{X}$ is called a \emph{closed solid cone in} $\mathcal{X}$ if $\mathcal{K}$ is closed in $\mathcal{X}$ and satisfies the following properties:
\begin{enumerate}[(i)]
    \item the interior of $\mathcal{K}$, $\interior \mathcal{K}$, is not empty,
    \item $\mathcal{K} + \mathcal{K} \subset \mathcal{K}$, 
    \item for all real $\lambda \geq 0$, $\lambda \mathcal{K} \subset \mathcal{K}$, and 
    \item $\mathcal{K} \cap -\mathcal{K} = \{\mathbf{0}\},$ where $\mathbf{0}$ is the zero element of $\mathcal{X}$. 
\end{enumerate}
\end{definition}

\begin{definition}[Hilbert's projective metric {\cite[Def.~2.1 and Def.~2.2]{Bus1973}}]
\label{def:Hilbert_metric}
Let $\mathcal{X}$ be a real Banach space, and let $\mathcal{K}$ be a closed solid cone in $\mathcal{X}$. The cone $\mathcal{K}$ induces a partial order relation $\preceq$ in $\mathcal{X}$ as follows: for any $\mathrm{\mathbf{x}},\mathrm{\mathbf{y}} \in \mathcal{X}$, 
\begin{equation}
\label{eq:def_partial_order_rel}
\mathrm{\mathbf{x}} \preceq \mathrm{\mathbf{y}} \,\Leftrightarrow\, \mathrm{\mathbf{y}} - \mathrm{\mathbf{x}} \in \mathcal{K}\,.     
\end{equation}
For any $\mathrm{\mathbf{x}}, \mathrm{\mathbf{y}} \in \mathcal{K}^{+} \eqdef \mathcal{K} \setminus \{\mathbf{0}\}$, define $M(\mathbf{x}, \mathbf{y}) \eqdef \inf \left\{ \lambda \mid \mathbf{x} \preceq \lambda \mathbf{y} \right\} $ and $m(\mathbf{x}, \mathbf{y}) \eqdef \sup \left\{ \lambda \mid \lambda \mathbf{y} \preceq \mathbf{x} \right\}$ with the convention $\inf \emptyset = \infty$\,. Then the \emph{Hilbert's projective metric} $d_H(\cdot,\cdot)$ is defined on $\mathcal{K}^{+}$ by \begin{equation}
\label{eq:Hilbert_proj_metric}
d_H(\mathbf{x}, \mathbf{y}) := \log\left( \frac{M(\mathbf{x}, \mathbf{y})}{m(\mathbf{x}, \mathbf{y})} \right).
\end{equation}   
\end{definition}
\begin{remark}
\label{rem:d_H_properties}
From Definition~\ref{def:Hilbert_metric}, it can be shown that for any $\mathbf{x}, \mathbf{y} \in \interior \mathcal{K}$ and any $\lambda, \mu > 0$, $d_H(\mathbf{x},\mathbf{y})$ is finite (see proof of Theorem 2.1 in \cite{Bus1973}) and $d_H(\lambda\mathbf{x},\mu\mathbf{y}) = d_H(\mathbf{x},\mathbf{y})$ (Lemma 2.2 in \cite{Bus1973}).  
\end{remark}

The Hilbert's projective metric $d_H$ restricted to a particular subset of $\mathcal{K}^{+}$ defines a metric space. This is given in the following proposition~\cite[Theorem 2.1]{Bus1973}. 
\begin{proposition}
\label{prop:metric_space}
Let $(\mathcal{X}, \|\cdot\|)$ be a real Banach space and $\mathcal{K}$ a closed solid cone in $\mathcal{X}$ with the partial order relation $\preceq$ in $\mathcal{X}$ induced by $\mathcal{K}$, defined in \eqref{eq:def_partial_order_rel}. Let $U$ denote the unit sphere in $\mathcal{X}$, that is, 
$
U \eqdef \{\,\mathbf{x} \in \mathcal{X} \,\,|\,\, \| \mathbf{x} \| = 1\, \}\,.  
$
Then $\left(\,U \cap \interior \mathcal{K}  \,,\,d_H\right)$ is a metric space. 
\end{proposition}
\section{Problem Formulation}
\label{sec:Problem_formulation}
In this section, inspired by the classical SBP on the Euclidean space, we consider an {isotropic} SBP for the kinematic equation $\dot{R}(t) = R(t) \hat{\Omega}(R,t)$ on the Lie group $\SO2$ as follows: 
\begin{problem}
\label{problem:SBP_on_SO(2)}
Let $\mathcal{U}$ be the set of admissible controls of finite energy to a system evolving on $\SO2$, that is, 
\begin{equation}
\label{eq:Controls}
\begin{aligned}
    \mathcal{U} \eqdef 
    \biggl\{ {\Omega} : \SO2 \times [0,1] &\rightarrow \mathbb{R} \;\bigg|\;
    \text{for each } t\in[0,1], \\[2pt]
    &\int_{\SO2} |\Omega(R,t)|^2\, d\mu(R) < \infty
    \biggr\}.
\end{aligned}
\end{equation}
Construct a controller $\Omega \in \mathcal{U}$ that solves the following minimum energy {stochastic optimal control problem}:
\begin{subequations}
\label{eq:SBP_on_SO(2)}
\begin{align}
\inf_{\Omega \in \mathcal{U}} \quad \quad &\mathbb{E} \left\{ \int_{0}^{1} \frac{1}{2} |\Omega(R,t)|^{2} \, \mathrm{d}t \right\} \label{1st_SBP_on_SO(2)} \\
\mathrm{subject\,\,to} \quad 
&\mathrm{d}R(t) = R(t)\, \widehat{\Omega}(R,t)\, \mathrm{d}t \notag \\ 
&\;\;\;\;\;\;\;\;\;\;+ \sigma\sum_{i=1}^{{m}} {R(t) B_i \; \circ \; }     \mathrm{d}w^i(t), \label{2nd_SBP_on_SO(2)} \\
&R(t=0) \sim \rho_{0}(R), \quad R(t=1) \sim \rho_{1}(R) \label{3rd_SBP_on_SO(2)}
\end{align}
\end{subequations}
where for each $i \in \{1,2, \dots ,m\}$, $B_i \in \mathfrak{so}(2)$ is constant, $w^i(t)$ is the standard Wiener process, and $\sigma > 0$ is the {diffusion strength}. The given initial and terminal joint state PDFs on $\SO2$, with finite second moments {with respect to the Haar (volume) measure $\mathrm{d}\mu(R)$}, are $\rho_0$ and $\rho_1$, respectively, i.e., $\rho_0, \rho_1 \in \mathcal{P}_2(\SO2)$ where
\begin{align}
\label{eq:P_2_SO(2)}
&\mathcal{P}_2(\SO2) \notag \\  
&\eqdef \biggl\{ \rho : \SO2 \rightarrow \mathbb{R}_{\geq 0} \;\bigg|\,
\int_{\SO2} \rho \, \mathrm{d}\mu(R) = 1, \notag \\ \;
&\;\;\;\;\;\;\;\;\;\;\;\;\;\;\;\;\;\;\;\;\;\;\;\;\;\;\;\;\;\;\;\;\;\;\;\;\;\,\int_{\SO2} {d(R,R_0)^2} \, \rho \, \mathrm{d}\mu(R) < \infty \biggr\},
\end{align} 
where $d(R,R_0)$ is the Riemannian distance of point $R \in \SO2$ from a fixed reference point $R_0 \in \SO2$. The expectation operator in \eqref{1st_SBP_on_SO(2)} is with respect to the controlled joint state PDF $\rho$\,, that is, $\mathbb{E}\{\cdot\} = \int_{\SO2} (\cdot) \, \rho \, \mathrm{d}\mu(R)$\,.
\end{problem}
Equivalently, we can describe the variational version of Problem \ref{problem:SBP_on_SO(2)} as follows:
\begin{subequations}
\label{eq:SBP_on_SO(2)_varprob}
\begin{align}
\inf_{\rho, \Omega} \quad \quad \quad & \int_0^1 \int_{\SO2} \frac{1}{2} |\Omega(R,t)|^2 \, \rho(R,t) \, \mathrm{d}\mu(R) \, \mathrm{d}t \label{1st_SBP_on_SO(2)_varprob} \\
\text{subject to} \quad &\partial_t \rho(R,t) = -\mathrm{div}(\rho(R,t) R(t) \widehat{\Omega}(R,t)) \notag \\
&\;\;\;\;\;\;\;\;\;\;\;\;\;\;\;\;\;\;\;\;\;\;\;\;\;\;\;\;\,+ \frac{\sigma^2}{2} \Delta_{\SO2} \rho(R,t),
\label{2nd_SBP_on_SO(2)_varprob} \\
&\rho(R,0) = \rho_0(R), \qquad \rho(R,1) = \rho_1(R), \label{3rd_SBP_on_SO(2)_varprob}
\end{align}
\end{subequations}
where \eqref{2nd_SBP_on_SO(2)_varprob} is the Fokker-Planck or Kolmogorov’s forward PDE associated with the controlled SDE \eqref{2nd_SBP_on_SO(2)}. The task is to find a controller $\Omega \in \mathcal{U}$ for the angular velocity of the system such that the PDF $\rho$, corresponding to $\Omega$, satisfies \eqref{2nd_SBP_on_SO(2)_varprob} and \eqref{3rd_SBP_on_SO(2)_varprob}, and this $\rho$ together with $\Omega$ minimizes the energy \eqref{1st_SBP_on_SO(2)_varprob}.  

\section{A Geometric Coordinate-free Solution}
\label{sec:schrodinger-bridge-geom}
In this section, we prove the existence and uniqueness of a solution to the Schrödinger system on $\SO2$. This will provide us with a geometric coordinate-free expression for the controller which will solve Problem~\ref{problem:SBP_on_SO(2)}.   
Consider the compact Lie group $\SO2$ with its Lie algebra $\mathfrak{so}(2)$. Consider an Ad-invariant inner product $\langle -,- \rangle_{\mathfrak{so}(2)}$ on $\mathfrak{so}(2)$ and let $\SO2$ be equipped with the induced bi-invariant Riemannian metric. 
 Let $\mu$ be the normalized Haar probability measure (the Riemannian volume), and $\Delta_{\SO2}$ the Laplace–Beltrami operator on $\SO2$. 
Fix the {diffusion strength} \(\sigma>0\) and put
\begin{equation}
\label{eq:gen_isotropic_Brownian_motion}
 L \;=\; \tfrac{\sigma^2}{2}\,\Delta_{\SO2},
\end{equation}
the generator of the isotropic Brownian motion on \(\SO2\).
Denote by \((T_t)_{0 \leq t \leq 1}= (e^{t L})_{0 \leq t \leq 1}\) the heat semigroup on \(L^2(\SO2,\mu)\) and by \(k_t:\SO2 \times \SO2 \to\mathbb R_{\geq 0}\) the heat kernel, i.e., for any $f\in L^2(\SO2)$ and for each $t \in [0,1]$,
\begin{equation}
\label{eq:(T_t f)(R)}
(T_t f)(R) \;=\; \int_{\SO2} k_t(R,\tilde{R})\,f(\tilde{R})\,d\mu(\tilde{R})\, .
\end{equation} 
Since $\SO2$ is connected, the heat kernel $k_t$ is strictly positive~\cite[Corollary 8.12]{Grig2012}. 
Let $C(\SO2)$ denote the space of real-valued continuous functions defined on $\SO2$. The space $C(\SO2)$ together with the sup-norm $\|\cdot\|_{\infty}$ (here, it is also a max-norm because $\SO2$ is compact) is a real Banach space~\cite{Piotr2007}.
 We now define the following sets:
\begin{align*}
    &C_{\geq 0}(\SO2) \eqdef \Bigl\{f \in C(\SO2) \Big|  \forall   R \in \SO2, f(R) \geq 0 \Bigr\}, 
    \\
    \notag\\
    &C_{+}(\SO2) \eqdef \Bigl\{ f \in C(\SO2) \Big|  \forall  R \in \SO2, f(R) > 0 \Bigr\}.
\end{align*} 

 \begin{remark}
 \label{rem:C+_nonempty}
 Note that $C_{+}(\SO2)$ is nonempty: the function $u : \SO2 \to \mathbb{R}$ defined by $u(R) = 1$ for all $R \in \SO2$ is clearly in $C_{+}(\SO2)$\,. 
 \end{remark}
 We now state a few important lemmas that we will need to prove our main result, Theorem \ref{thm:Soln_Schrödinger_sys_SO2}. Their proofs can be found in the Appendix~\ref{sec:Appendix}.

\begin{lemma}[Interior of \(C_{\geq 0}(\SO2)\)]
\label{lem:intC{geq 0}}
    In the real Banach space $\left(C(\SO2), \|\cdot\|_{\infty} \right)$, we have  
\( \interior C_{\geq 0}(\SO2) = C_{+}(\SO2).\)
\end{lemma}
With Lemma~\ref{lem:intC{geq 0}}, it follows from Definition~\ref{def:closed_solid_cone} that $C_{\geq 0}(\SO2)$ is a closed solid cone in the Banach space $\left(C(\SO2), \|\cdot\|_{\infty} \right)$. Let $U$ be the unit sphere in this Banach space. From Lemma~\ref{lem:intC{geq 0}} and Proposition~\ref{prop:metric_space}, we see that $\left(C_{+}(\SO2) \cap U, d_H\right)$ is a metric space. The following lemma states that this metric space is complete.

\begin{lemma}
\label{lem:Complete_metric_space}
    Let $U$ be the unit sphere in the Banach space $\left(C(\SO2), \|\cdot\|_{\infty} \right)$, that is,  
        \( U \eqdef \{\, f \in C(\SO2) \,\,| \,\, \| f\|_{\infty} = 1\,\}.\)
Then $\left(C_{+}(\SO2) \cap U, d_H\right)$ is a complete metric space.  
\end{lemma}

\begin{lemma}
\label{lem:R_f_isometry}
Fix any $f \in C_{+}(\SO2)$\,. The map $R_f : C_{+}(\SO2)   \to C_{+}(\SO2)$ defined by  
\(R_f(g) \eqdef {f}/{g}\) 
(\,with $f/g$ understood as $\left(f/g\right)(x) = f(x)/g(x)$\,) is an isometry with respect to the Hilbert's projective metric $d_H$, that is, for any $g_1, g_2 \in C_{+}(\SO2) $\,, 
$
    d_H(R_f(g_1), R_f(g_2)) = d_H(g_1, g_2).
$
\end{lemma}

\begin{remark}
\label{rem:T_1_Dom&Codom_C+}
Since the heat kernel $k_t$ is strictly positive on $\SO2 \times \SO2$, it follows from \eqref{eq:(T_t f)(R)} that for any $f \in C_{+}(\SO2)$, we have for all $R \in \SO2, (T_1 f)(R) > 0$. Also, for any $f \in L^{2}(\SO2)$ and $t>0$, we have $T_t f \in C^{\infty}(\SO2)$ ~\cite[Theorem 7.6]{Grig2012}. Hence, for any $f \in C_{+}(\SO2)$, we get $T_1 f \in C_{+}(\SO2)$.
\end{remark}
With the above remark in mind, we have the following important result:
\begin{lemma}
\label{lem:T_1_str_contraction}
There exists a constant $c \in [0,1)$ such that for all $f_1, f_2 \in C_{+}(\SO2)$\,, we have  
$
d_H(T_1 f_1, T_1 f_2) \leq c \,\, d_H(f_1, f_2).
$
\end{lemma}

\subsection{Conditions for Optimality and the Schrödinger System}
To obtain the first-order conditions for optimality, we start by considering the Lagrangian associated with \eqref{eq:SBP_on_SO(2)_varprob}
\begin{align}
\label{eq:Lagrangian}
\mathcal{L}(\rho,\Omega,\Psi) \;=\;\int_0^1\!\!\int_{\SO2} \Bigl\{ &\tfrac{1}{2}|\Omega|^2 \rho
\;+\;\Bigl(\partial_t \rho  +\operatorname{div}(\rho R \widehat{\Omega}) \notag \\ 
&- \tfrac{\sigma^2}{2}\Delta_{\SO2} \rho\Bigr)\Psi\Bigr\} \,d\mu \,dt,
\end{align}
where $\Psi : \SO2 \times [0,1] \to \mathbb{R}$ is a smooth Lagrange multiplier. Next, we define 
\begin{align}
\label{eq:P_01(SO2)}
&\mathcal{P}_{01}(\SO2) \notag \\ 
&\eqdef \biggl\{ \rho: \SO2 \times [0,1] \to \mathbb{R}_{\geq 0} \,\bigg|\, \rho(\cdot,0) = \rho_0, \rho(\cdot,1) = \rho_1,  \notag \\ 
&\;\;\;\;\;\text{for each }t \in [0,1],\int_{\SO2} \rho(R,t)\, d\mu(R) = 1\biggr\}.
\end{align}
Performing the unconstrained minimization of the Lagrangian $\mathcal{L}$ over $\mathcal{P}_{01}(\SO2) \times \mathcal{U}$, we find that the optimal pair $(\rho^{\mathrm{opt}}, \Omega^{\mathrm{opt}})$ that solves \eqref{eq:SBP_on_SO(2)_varprob} must satisfy the following system of coupled PDEs
\begin{subequations}
\label{eq:rhoOpt&psi_SO2}
\begin{align}
&\partial_t\Psi + \frac{1}{2}\left|\left(R(t)^{\top}\grad\Psi\right)^{\vee}\right|^2 = -\,\frac{\sigma^2}{2}\,\Delta_{\SO2}\Psi, 
\label{eq:1st_rhoOpt&psi_SO2} \\ 
&\partial_t\rho^{\mathrm{opt}} + \operatorname{div}(\rho^{\mathrm{opt}}\grad\Psi) = \frac{\sigma^2}{2}\,\Delta_{\SO2}\rho^{\mathrm{opt}}, 
\label{eq:2nd_rhoOpt&psi_SO2} \\
&\rho^{\mathrm{opt}}(\cdot,0)=\rho_0,\quad \rho^{\mathrm{opt}}(\cdot,1)=\rho_1,
\label{eq:3rd_rhoOpt&psi_SO2}
\end{align}
\end{subequations}
with 
\begin{equation}
\label{eq:Omega^{opt}} 
\Omega^{\mathrm{opt}}(R,t) = \left(R^{\top}\grad\Psi\left(R,t\right)\right)^{\vee}.
\end{equation}
Using the Hopf--Cole transform $(\rho^{\text{opt}}, \Psi) \mapsto (\varphi, \psi)$ given by 
\begin{subequations}
\label{eq:Hopf-Cole}
\begin{align}
&\varphi(R,t)=\exp\left(\frac{\Psi(R,t)} {\sigma^2}\right),\label{1st_Hopf-Cole}\\ 
&\psi(R,t) = \rho^{\mathrm{opt}}(R,t) \exp\left(-\frac{\Psi(R,t)}{\sigma^2}\right) 
\label{2nd_Hopf-Cole}
\end{align}
\end{subequations}
transforms the 
system \eqref{eq:rhoOpt&psi_SO2} into the pair of linear (forward and backward) heat equations
\begin{equation}
\label{eq:forward_backward_heat_eqns}
\partial_t\varphi = -\frac{\sigma^2}{2}\,\Delta_{\SO2}\varphi\,,
\quad \partial_t\psi = \frac{\sigma^2}{2}\,\Delta_{\SO2}\psi
\end{equation}
with coupled boundary conditions 
\begin{subequations}
\label{eq:coupled_bnd_cond_SO2}
\begin{align}
\varphi(\cdot,0)\,\psi(\cdot,0)=\rho_0\,,
\label{1st_coupled_bnd_cond_SO2} \\ 
\varphi(\cdot,1)\,\psi(\cdot,1)=\rho_1\,. 
\label{2nd_coupled_bnd_cond_SO2}
\end{align}
\end{subequations}
For notational simplicity, let 
\(\varphi_1 \eqdef \varphi(\cdot, t=1), \;\; \psi_0 \eqdef \psi(\cdot, t=0)\).
Then, we can express the solution of forward-backward heat equations \eqref{eq:forward_backward_heat_eqns} on $\SO2$ using the heat semigroup \eqref{eq:(T_t f)(R)} as \begin{subequations}
\label{eq:soln_heat_eqns_SO2}
\begin{align}
&\varphi(R,t) := (T_{1-t}\varphi_1)(R)\,, \label{1st_soln_heat_eqns_SO2} \\ 
&\psi(R,t) := (T_t\psi_0)(R)\,, 
\label{2nd_soln_heat_eqns_SO2}
\end{align}
\end{subequations} for all $R \in \SO2$ and $t \in [0,1]$.
Plugging the expression for $\varphi(\cdot,0)$ from \eqref{1st_soln_heat_eqns_SO2} into \eqref{1st_coupled_bnd_cond_SO2} and the expression for $\psi(\cdot,1)$ from \eqref{2nd_soln_heat_eqns_SO2} into \eqref{2nd_coupled_bnd_cond_SO2}, we get the following system of equations called the \emph{Schrödinger system on $\SO2$}:  
\begin{equation}
\label{eq:Schrödinger_sys_SO2}
\rho_0=\psi_0\,T_{1}\varphi_1\,,
\qquad \rho_1=\varphi_1\, T_1\psi_0\,.
\end{equation}
We shall now prove, as our main result of this paper, the existence of a solution $(\varphi_1,\psi_0)$ to this Schrödinger system on $\SO2$ in the following theorem. This solution is unique up to multiplication of $\varphi_1$ and division of $\psi_0$ by the same positive constant. Also, we restrict the time-1 heat operator \(T_1\) to the subset $C_{+}(\SO2)$ of $L^{2}(\SO2)$ and also take $C_{+}(\SO2)$ as the codomain of $T_1$, which is allowed because of Remark~\ref{rem:T_1_Dom&Codom_C+}. 

\begin{theorem}[Solution of the Schrödinger system on $\SO2$]\label{thm:Soln_Schrödinger_sys_SO2}
For the Lie group $\SO2$, let $T_1 : C_{+}(\SO2) \to C_{+}(\SO2)$\, be the time-1 heat operator and let $k_1$ be the smooth, strictly positive heat kernel of\/ $T_1$.  Assume that the initial and terminal PDFs $\rho_0$ and $\rho_1$, respectively, satisfy $\rho_0 \,,\, \rho_1 \in C_+(\SO2)$. Then there exist \(\varphi_1\,, \psi_0 \,\in C_+(\SO2)\) satisfying the following Schrödinger system of equations: 
\begin{equation}
\label{eq:Schrödinger_system}
\qquad
\psi_0 \;=\; \frac{\rho_0}{T_1\varphi_1},\qquad
\varphi_1 \;=\; \frac{\rho_1}{T_1\psi_0}.
\qquad
\end{equation}
Moreover, this solution is unique up to multiplication of $\varphi_1$ and division of $\psi_0$ by the same positive constant. 
\end{theorem}
\begin{proof}
Plugging the expression for $\psi_0$ from the first equation in \eqref{eq:Schrödinger_system} into the second, we get  
\begin{equation}
\label{eq:varphi_eq}
\varphi_1 \;=\; \frac{\rho_1}{T_1\left(\frac{\rho_0}{T_1(\varphi_1)}\right)}.
\end{equation}Looking at \eqref{eq:varphi_eq}, we define a map $F :  C_{+}(\SO2) \to C_{+}(\SO2) $ as follows: 
$
    F \eqdef R_{\rho_{1}} \circ T_{1} \circ R_{\rho_{0}} \circ T_{1} ,
$
where the maps $R_{\rho_{0}}$ and $R_{\rho_{1}}$ are defined in Lemma~\ref{lem:R_f_isometry}. Then for any $f_1, f_2 \in C_{+}(\SO2)$, we have the following:
\begin{align}
    &d_H(F(f_1),F(f_2)) \notag \\ 
    &= d_H\Big(\,R_{\rho_{1}} \big(\,T_{1} ( R_{\rho_{0}} ( T_{1}
(f_1)))\,\big) \, , \, R_{\rho_{1}}\big(\,T_{1} ( R_{\rho_{0}} ( T_{1}
(f_2)))\,\big)\,\Big) \notag \\
    &\overset{\mathrm{Lemma}~\ref{lem:R_f_isometry}}{=} d_H\Big(\,T_{1}\big(\, R_{\rho_{0}} ( T_{1}
    (f_1))\,\big)\, , \,T_{1}\big(\,R_{\rho_{0}} ( T_{1}
    (f_2))\,\big)\,\Big) \notag \\
    &\overset{\mathrm{Lemma}~\ref{lem:T_1_str_contraction}}{\leq} c\,\,d_H\Big(\, R_{\rho_{0}}\big(\, T_{1}
    (f_1)\,\big)\, , \,R_{\rho_{0}}\big(\, T_{1}
    (f_2)\,\big)\,\Big) \notag \\
    &\overset{\mathrm{Lemma}~\ref{lem:R_f_isometry}}{=} c\,\,d_H(T_{1}
    (f_1), T_{1}
    (f_2)) \notag \\
    &\overset{\mathrm{Lemma}~\ref{lem:T_1_str_contraction}}{\leq} c^2\,\,d_H(f_1, f_2).
    \label{F_strict_contraction}
\end{align}
Since $c \in [0,1)$ (from Lemma~\ref{lem:T_1_str_contraction}) implies $c^2 \in [0,1)$, the above means that the map $F$ is a strict contraction on $C_{+}(\SO2)$ with respect to the Hilbert's projective metric $d_H$\,. 

Let $U$ denote the unit sphere in the Banach space $\left(C(\SO2), \|\cdot\|_{\infty} \right)$ and define a map $\tilde{F} : C_{+}(\SO2) \cap U \to C_{+}(\SO2) \cap U$ by   
$
\tilde{F}(\phi) \eqdef \frac{F(\phi)}{\| F(\phi) \|_{\infty}}.
$ Note that for any $f_1, f_2 \in  C_{+}(\SO2) \cap U $, we have    
\begin{align*}
d_H\left(\tilde{F}(f_1), \tilde{F}(f_2)\right) &= d_H\left(\frac{F(f_1)}{\| F(f_1) \|_{\infty}} , \frac{F(f_2)}{\| F(f_2) \|_{\infty}}\right) \\ 
&\overset{\mathrm{Remark}~\ref{rem:d_H_properties}}{=} d_H\left(F(f_1), F(f_2)\right) \\ 
&\overset{\mathrm{by\,}\eqref{F_strict_contraction}}{\leq} c^2\,\,d_H(f_1, f_2).  
\end{align*}
With $c^2 \in [0,1)$, the above means that $\tilde{F}$ is also a strict contraction on its domain $C_{+}(\SO2)\, \cap\, U$ with respect to $d_H$\,.

Now, consider any function $\phi_0 \in C_{+}(\SO2)$. We can choose this function in $C_{+}(\SO2)$ because $C_{+}(\SO2)$ is nonempty (see Remark~\ref{rem:C+_nonempty}). We then define $\phi_k \eqdef \tilde{F}^{k}\left(\frac{\phi_0}{\| \phi_0 \|_{\infty}}\right), \,k \geq 1$, where $\tilde{F}^{k} = \underbrace{ \tilde{F} \circ \tilde{F} \circ ... \circ \tilde{F}}_{k}$\,. Since $\tilde{F}$ is, with respect to $d_H$\,, a strict contraction on $C_{+}(\SO2)\, \cap\, U$ which is a complete metric space with respect to $d_H$ (see Lemma~\ref{lem:Complete_metric_space}), by the Banach fixed point theorem, the map $\tilde{F}$ has a unique fixed point in $C_{+}(\SO2)\, \cap\, U$, that is, there exists a unique function $\phi^{*} \in C_{+}(\SO2)\, \cap\, U$ such that $\tilde{F}(\phi^{*}) = \phi^{*}$\,. This $\phi^{*}$ is the limit of the sequence $\{\phi_{k}\}_{k \geq 1}$ in $\left(C_{+}(\SO2)\, \cap\, U\,,\,d_H\right)$, that is, $\phi_k \xrightarrow[]{d_H} \phi^{*}$ in $C_{+}(\SO2)\, \cap\, U$\,, which means that $\lim_{k \to \infty} d_H(\phi_k,\phi^{*}) = 0$.

We now finally show that $F(\phi^{*}) = \phi^{*}$\,. Since $\tilde{F}(\phi^{*}) = \phi^{*}$\,, we get $F(\phi^{*}) = \lambda\, \phi^{*}$\,, where $\lambda = \| F(\phi^{*}) \|_{\infty} > 0$\,. Define the inner product $\langle\cdot,\cdot\rangle : C(\SO2) \times C(\SO2) \to \mathbb{R}$ by 
$
\langle f,g\rangle \eqdef \int_{\SO2} f(R)\,g(R)\,d\mu(R).
$ 
Since $\langle \phi, R_{\rho_{1}}(\phi) \rangle = \int_{\SO2} \rho_1(R)\, d\mu(R) = 1$ for any $\phi \in C_{+}(\SO2)$\,, we have the following: 
\begin{align*}
1 &= \int_{\SO2} \rho_1 d\mu(R) \\  
  &= \langle \,\,(T_1 \circ R_{\rho_0} \circ T_1)(\phi^{*}) \,\,,\, (\underbrace{R_{\rho_1} \circ T_1 \circ R_{\rho_0} \circ T_1}_{F})(\phi^{*}) \,\,\rangle \\
  &= \langle \,\,(T_1 \circ R_{\rho_0} \circ T_1)(\phi^{*}) \,\,,\, \lambda\, \phi^{*}\,\,\rangle \\
  &=\langle\,\, (R_{\rho_0} \circ T_1)(\phi^{*}) \,\,,\,T_1(\lambda \phi^{*}) \,\,\rangle \\
  &=\langle\,\, R_{\rho_0}\left( T_1(\phi^{*})\right) \,\,,\,\lambda\,T_1(\phi^{*}) \,\,\rangle \\
  &=\int_{\SO2} \rho_0(R)\, \lambda\,d\mu(R) 
  =\lambda \cdot 1 = \lambda\,,  
\end{align*}
where in the third step above, we used $F(\phi^{*}) = \lambda\, \phi^{*}$.
Hence, $F(\phi^{*}) = \lambda\, \phi^{*} = \phi^{*}$\,, and so, we have  
\begin{align}
\label{eq:sol_varphi}
 \phi^{*} &= F(\phi^{*})  
 = \frac{\rho_1}{T_1\left( \frac{\rho_0}{T_1(\phi^{*})}\right)}.
\end{align} Comparing \eqref{eq:sol_varphi} with \eqref{eq:varphi_eq}, we see that $\phi^{*}$ solves \eqref{eq:varphi_eq}, and so, 
\begin{equation}
\label{eq:Schrödinger_sys_soln} \qquad
\varphi_1 \;=\; \phi^{*},\qquad
\psi_0 \;=\; \frac{\rho_0}{T_1\phi^{*}}\qquad
\end{equation} is a solution of the Schrödinger system \eqref{eq:Schrödinger_system}. The uniqueness of this solution up to multiplication of $\varphi_1$ and division of $\psi_0$ by the same positive constant follows from the uniqueness of $\phi^{*}$. This completes the proof.  
\end{proof}

\begin{remark}
With the solution $(\varphi_1,\psi_0)$ (to the Schrödinger system \eqref{eq:Schrödinger_sys_SO2} on $\SO2$) obtained, we use \eqref{eq:soln_heat_eqns_SO2} to get $\varphi(R,t)$ and $\psi(R,t)$. Then, multiplying \eqref{1st_Hopf-Cole} with \eqref{2nd_Hopf-Cole} gives  
\begin{equation}
\label{eq:rho_Opt}
\rho^{\mathrm{opt}}(R,t) = \varphi(R,t)\,\psi(R,t).
\end{equation}
Also, now with $\varphi(R,t)$ known, we have from \eqref{1st_Hopf-Cole}  
$
\Psi(R,t) = \sigma^2 \log \varphi(R,t). 
$
Plugging this expression for $\Psi(R,t)$ into \eqref{eq:Omega^{opt}}, we finally have 
\begin{equation}
\label{eq:Omega^{opt}_soln}
\Omega^{\mathrm{opt}}(R,t) = \left(\sigma^2 R^{\top}\grad\log \varphi(R,t) \right)^{\vee}.
\end{equation}
This controller $\Omega^{\mathrm{opt}}$ is the geometric coordinate-free solution to our SBP for the kinematic equation on $\SO2$ (Problem~\ref{problem:SBP_on_SO(2)}).    
\end{remark}

\section{Numerical Simulation}
\label{sec:numerical_simulations}
\begin{figure}
    \centering
    \includegraphics[scale = 0.3]{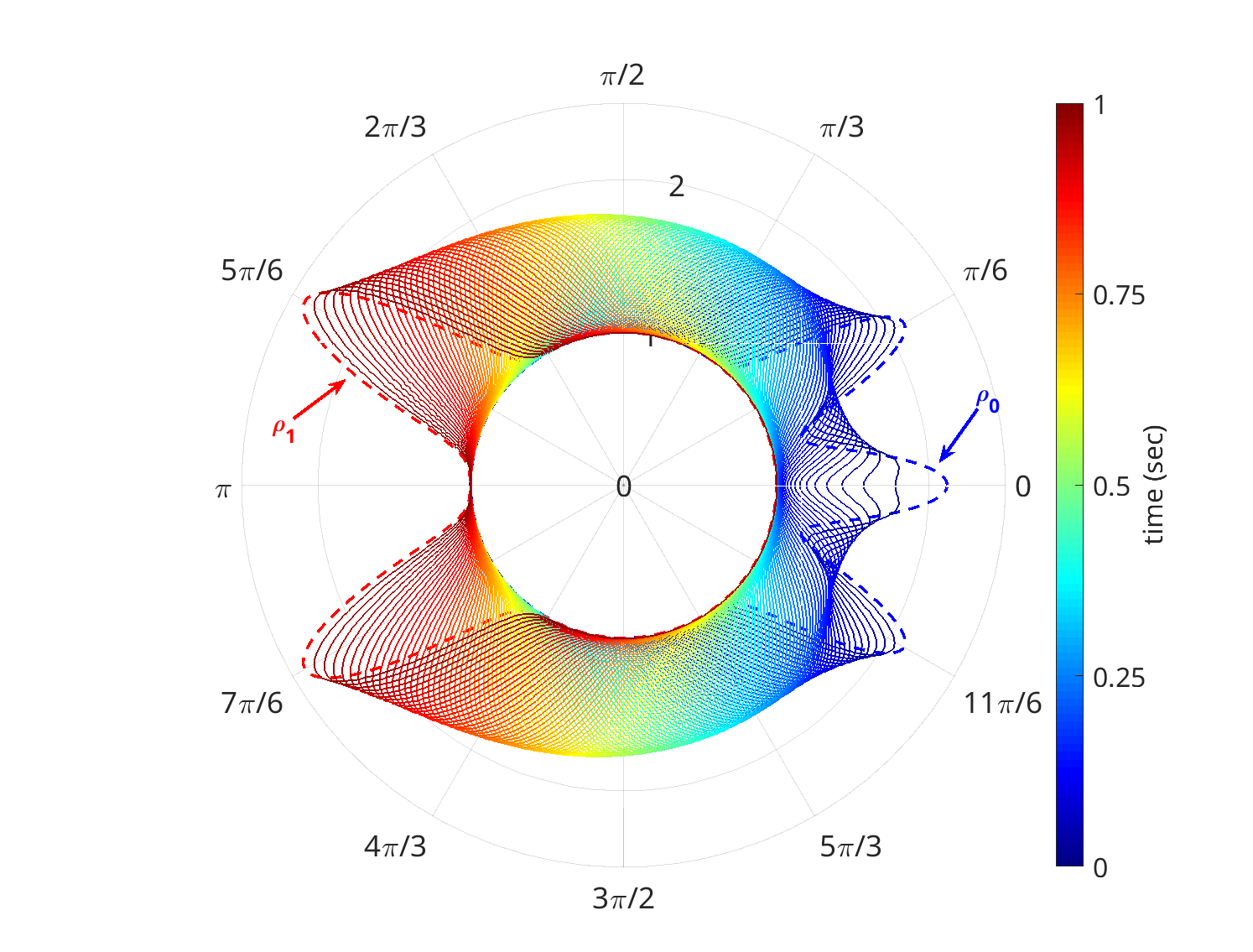}
    \caption{SBP with initial PDF of three peaks and terminal PDF of two peaks}
    \label{fig:SBP-peaks_threeTOtwo}
    \vspace{0.2in}
\end{figure}
In this section, we present a numerical simulation of SBP for the kinematic equation on $\SO2$ using a stabilized (log-domain) version of the Sinkhorn (or IPFP) algorithm implemented via the Fast Fourier Transform (FFT). 
Because of space limitations, we provide only one simulation result; more tested examples can be found on the \texttt{Git}~\footnote{\label{footnote:code_link}
Code and animations at 
\href{https://gitlab.com/a5akhtar/sbp}
{https://gitlab.com/a5akhtar/sbp}} page.

The circle $\SO2$ is discretized uniformly into $N=1024$ grid points; the choice
of such a value of $N$ prevents
spurious oscillations arising from the Gibbs phenomenon. 
The isotropic heat kernel $K_t$ on $\SO2$ admits the following Fourier representation:
$
K_t(\theta) = \frac{1}{2\pi}\sum_{k\in\mathbb{Z}} e^{-\frac{1}{2}\sigma^2 k^2 t} e^{ik\theta},
$
where $\sigma = 0.43$ is the diffusion strength. Numerically, the convolution ($*$) with $K_t$ can be
realized via Fast Fourier Transform (FFT) as multiplication by spectral
multipliers $\lambda_t(k) = \exp(-\tfrac{1}{2}\sigma^2 k^2 t)$.

The intermediate PDFs $\rho_t(\theta)$ in Figure~\ref{fig:SBP-peaks_threeTOtwo} are plotted on a polar coordinate system, where the value of the PDF at any angle $\theta$ is represented by the height of the curve above the unit circle at that angle.  
The simulation test considers the von Mises distribution whose PDF is given by   
$
\rho(\theta) = \frac{1}{2\pi I_0(\kappa)} \, e^{\kappa \cos(\theta - \mu)},
$
where $\mu$ is the mean of the distribution, $\kappa > 0$ is a concentration parameter, and $I_0(\kappa)$ is the modified Bessel function of the first kind of order 0. The evolution of the intermediate densities in Fig.~\ref{fig:SBP-peaks_threeTOtwo} illustrates how the three peaks smoothly redistribute to become two under the Schrödinger bridge dynamics on $\SO2$. For further details, readers are referred to our $\texttt{Git}^{~\ref{footnote:code_link}}$ page. 

\section{Conclusion}
\label{sec:Conclusion}
In this work, we developed a geometric, coordinate-free solution to the isotropic Schrödinger bridge problem for the kinematic equation on the Lie group $\SO2$. By formulating the problem as a stochastic optimal control problem with angular velocity as the control input, we established existence and uniqueness of a solution to the associated Schrödinger system via a contractive fixed-point recursion with respect to the Hilbert’s projective metric. The proposed approach exploits only the intrinsic geometric structure of $\SO2$, avoiding any embedding in the Euclidean plane. The numerical simulation demonstrates that the algorithm successfully steers probability densities between prescribed marginals, thereby validating the geometric solution of the Schrödinger system on the unit circle.

\section{Appendix}
\label{sec:Appendix}
\subsection{Proofs}
\begin{proof}[Proof of Lemma~\ref{lem:intC{geq 0}}]
See the proof of Lemma 1 in~\cite{mahmood2026}.
\end{proof}
\begin{proof}[Proof of Lemma~\ref{lem:Complete_metric_space}]
See Theorem 4.1, Theorem 4.2, and Theorem 4.4 in~\cite{Bus1973}.    \end{proof}

\begin{proof}[Proof of Lemma~\ref{lem:R_f_isometry}]
Since multiplication by a fixed $f \in C_{+}(\SO2)$ is an isometry with respect to the Hilbert's projective metric $d_H$ and $M(\mathbf{x}_1,\mathbf{x}_2) = (m(\mathbf{x}_1^{-1},\mathbf{x}_2^{-1} ))^{-1}$, the statement of the Lemma follows.
\end{proof}

\begin{proof}[Proof of Lemma~\ref{lem:T_1_str_contraction}]
The proof is similar to that of Lemma 5 in~\cite{Che2016}. 
\end{proof}

\bibliographystyle{IEEEtran}
\bibliography{myreferences}

\end{document}